\documentclass{amsart}
\usepackage{amssymb}
\usepackage{mathrsfs}
\usepackage{mathdots}
\usepackage{hyperref}
\usepackage{verbatim}
\usepackage{colortbl}
\usepackage{graphicx}
\usepackage{subfigure}
\usepackage{psfrag}
\usepackage{enumitem}

\newtheorem{theorem}{Theorem}[section]

\newtheorem{corollary}[theorem]{Corollary}
\newtheorem{proposition}[theorem]{Proposition}

{\theoremstyle{definition}

\newtheorem{remark}{Remark}}

\newcommand{\R}{\mathbb{R}}

\begin{document}

\title[non-injective polynomial local diffeomorphisms]{A new class of non-injective polynomial local diffeomorphisms on the plane}

\author[F. Fernandes]
{{Filipe Fernandes}}

\address{Departamento de Matem\'{a}tica, Universidade Federal de S\~ao Carlos, 
13565--905 S\~ao Carlos, S\~ao Paulo, Brazil}
\email{filipefernandes@dm.ufscar.br} 



\date{\today}

\begin{abstract}
In this short note we provide the example with the lowest degree known so far of a non-injective local polynomial diffeomorphism $F=(p,q):\mathbb{R}^2 \to \mathbb{R}^2$. In our example $p$ has degree $10$ and $q$ has degree $15$, rather than $10$ and $25$, respectively, known up to now as the smallest degrees for the coordinates of $F$.
Our construction was based on S. Pinchuk celebrated counterexample to the real Jacobian conjecture.
\end{abstract}

\maketitle

\section{Introduction} 
The \emph{real Jacobian conjecture} claims that a polynomial local diffeomorphism $F = (p,q): \R^2 \to \R^2$ must be injective. It was proved false by means of a class of counterexamples, the \emph{Pinchuk maps}, constructed by Pinchuk [4]. A Pinchuk map $F = (p,q)$ is such that the polynomial $p$ is always the same and the polynomial $q$ varies as the sum of a fixed polynomial with another one that satisfies a suitable differential equation. The degree of $p$ is $10$ and it is known that the smallest possible degree of $q$ in a Pinchuk map is $25$ \cite{campbell}. 

On the other hand it is known that any polynomial local diffeomorphism $F=(p,q)$ with $\deg(p) \leq 4$ satisfies the real Jacobian conjecture \cite{BOO,BS,Gw}. In this note we provide a new class of counterexamples to the real Jacobian conjecture, where the polynomial $p$ is the same as in a Pinchuk map, but now it is possible to take $q$ with degree $15$.

\section{Construction of the Example}

We consider the same three auxiliary polynomials as in Pinchuk's paper:
$$
t=xy-1,\ \ \ \ h=xt^2+t, \ \ \ \ f=\frac{h^3+h^2}{h-t}=(xt+1)^2(t^2+y).
$$
From now on we denote by $J(a,b)$ the Jacobian determinant of a map $(a,b): \mathbb{R}^2 \to \mathbb{R}^2$, that is,
$$
J(a,b)=\frac{\partial a}{\partial x}\frac{\partial b}{\partial y}-\frac{\partial a}{\partial y}\frac{\partial b}{\partial x}.
$$
As in \cite[Lemma 2.1]{pinchuk}:
\begin{proposition}\label{jhtjfh}
The following differential properties are true:
\begin{enumerate}[label={\textnormal{(\alph*)}}]
\item\label{jht} $J(h,t)=h-t,$
\item\label{jfh} $J(f,h)=-f$
\end{enumerate}
\end{proposition}
\begin{proof}
Both items follow from straightforward computations using the chain and the product rules and the linearity of differentiation, as well as the usual determinant properties. For item \ref{jht}, we have 
$$J(h,t)=J(xt^2+t,t)=t^2J(x,t)=xt^2=h-t.$$
For item \ref{jfh}, 
$$
\begin{aligned}
J(f,h) =J\left(\frac{h^3+h^2}{h-t},h\right)&=(h^3+h^2)J\left((h-t)^{-1},h\right) \\
&=-\frac{h^3+h^2}{(h-t)^2}J(-t,h)\\
&=-f,
\end{aligned}
$$
where the last equality follows from item \ref{jht}.
\end{proof}

As in Pinchuk example, we take the first coordinate of our example as
$$
p = f +h
$$
which can be shown being, up to scaling, the only linear combination of f and h that has a non-vanishing gradient.
Moreover, since the level set $p=0$ is disconnected, as $xt+1$ is a factor of $p$ and $xt+1=0$ is equivalent to $y=(x-1)/x^2$ that has two connected components, it follows that for any $q$ such that $J(p,q)>0$, the map $(p,q)$ is not a global diffeomorphism, and so it is not injective as we can conclude from the main result of \cite{Rosenlicht}.

\begin{proposition}\label{jp}
Let $m$ be an integer and $G$ be a real differentiable function, the following holds:
$$
J(p,G(h)f^m)=-f^{m+1}G'(h)+mf^mG(h)
$$
\end{proposition}
\begin{proof}
Acting analogously as in the proof of Proposition \ref{jhtjfh},
$$
\begin{aligned}
J(p,G(h)f^m) &=J(f,G(h)f^m)+J(h,G(h)f^m)\\
&=f^mJ(f,G(h))+G(h)J(h,f^m)\\
&=f^mG'(h)J(f,h)+mG(h)f^{m-1}J(h,f),
\end{aligned}
$$
and the claimed identity now follows from item \ref{jfh} of Proposition \ref{jhtjfh}.
\end{proof}
The next step in the construction is to take $q$ as
\begin{equation}\label{q}
q=\sum _{i=-2}^{1}{f}^{i}M_i(h),
\end{equation}
where $M_i(h), i\in \{-2,-1,0,1\}$, are differentiable functions. Choosing $q$ in this way is convenient because after computing $J(p,q)$, we can find conditions on $M_i(h)$ for the Jacobian to be positive and for $q$ to be a polynomial.

\begin{proposition}\label{jpq}
Let $q$ as in \eqref{q}, then
\begin{equation}\label{jjpq}
J(p,q)=-2\,{\frac {M_{{-2}} \left( h \right) }{{f}^{2}}}+\sum _{i=-1}^{1}{f}^
{i} \left( iM_{{i}} \left( h \right) -M_{{i-1}}' \left( h \right) 
 \right) -M_{{1}}' \left( h \right) {f}^{2}.
\end{equation}
\end{proposition}
\begin{proof}
By Proposition \ref{jp} we have
$$
J(p,q)=\sum _{i=-2}^{1}J\left(p,{f}^{i}M_i(h)\right)=\sum _{i=-2}^{1}\left(-f^{i+1}M_i'(h)+if^iM_i(h)\right),
$$
and the result follows from collecting the powers of $f$.
\end{proof}
Our goal now is to find hypothesis on $q$ in order to have $J(p,q)>0$. In his construction, Pinchuk uses the fact that $f$ and $t$ do not vanish simultaneously to construct $q$ in such a way that $J(p,q)=t^2+f^2+M(f,h,t)^2>0$ where $M$ is a suitable polynomial. Instead, by using Groebner basis we looked for other algebraic combinations of $t$ and $h$ in a way that these combinations did not vanish simultaneously with $f$. We found out that $f$ and $$c(h^2+h)-t=\frac{h^3+h^2}{f}-h+c(h+h^2)$$ have this property, where $c\in \mathbb{R}$:
\begin{proposition}\label{vanish}
Let $c\in \mathbb{R}$, then $f$ and  $c(h^2+h)-t$ do not vanish simultaneously.
\end{proposition}
\begin{proof}
By definition, if $f=0$, we have that $xt+1=0$ or $t^2+y=0$. In both cases $t\neq0$. However, if $xt+1=0$ then $h=0$, and if $t^2+y=0$ then $h+1=0$. In both cases $c(h^2+h)=0$ concluding that $f$ and  $c(h^2+h)-t$ do not vanish simultaneously.
\end{proof}
Next theorem provides conditions on $q$, defined by \eqref{q}, in a way that
\begin{equation}\label{jjpq2}
J(p,q)=\left( {\frac {h^3+h^2}{f}}-h+c \left( {h}^{2}+h \right) 
 \right) ^{2}+ \left( {\frac {h^3+h^2}{f}}+{N_0}
 \left( h \right) +{N_1} \left( h \right) f \right) ^{2}+{f}^{2},
\end{equation}
where $N_0$ and $N_1$ are suitable differentiable functions, and therefore $J(p,q)>0$. It will also show conditions for $q$ to be a polynomial.
\begin{theorem}\label{jacobian}
Let $q$ be as in \eqref{q}, and $N_0$ a differentiable function, defining
\begin{align*}
M_{{-2}} \left( h \right) &=- \left( {h}^{3}+{h}^{2} \right) ^{2},\\
M_{{-1}} \left( h \right)&=2\left( h^3+h^2 \right)  \left( (3-c)(h^2+h)-{N_0}
 \left( h \right) \right),\\
{N_1} \left( h \right) &=-{\frac {{M_{-1}'}
 \left( h \right) +{h}^{2} \left( ch+c-1 \right) ^{2}+ {N_0}
 \left( h \right) ^{2}}{2({h}^{3}+{h}^{2})}},\\
M_{{1}} \left( h \right) &=-\int_{0}^{h}\! {N_1} \left( s
 \right)  ^{2}\,{\rm d}s-h,\\
 M_{{0}} \left( h \right) &=\int_{0}^{h}\!\big(M_{{1}} \left( s \right)-2\,{N_0} \left( s \right) 
{N_1} \left( s \right)\big) \,{\rm d}s,
\end{align*}
then $J(p,q)$ satisfies \eqref{jjpq2}. Moreover, given any $K(h)\in\mathbb{R}[h]$, if we additionally assume
$$N_0(h)=-h+(h^2+h)K(h),$$
then $q$ is a polynomial in $(x,y)$.
\end{theorem}
\begin{proof}
By acting as if $f$ and $h$ are independent variables, it suffices to compare the coefficients of $f^m$ in \eqref{jjpq} and \eqref{jjpq2} to verify the first part of the result. 
For the second part we first observe that $M_{-2}(h)/f^2$ and $M_{-1}(h)/f$ are polynomials in $(x,y)$ since $(h^3+h^2)/f$ is a polynomial in $(x,y)$. Now by unwinding $M_{-1}'(h)$ and completing squares of $N_0(h)$ on the $N_1$ expression it follows that
$$
{N_1} \left( h \right) =-{\frac { \left( 3\,{h}^{2}+2\,h-{N_0} \left( h \right)  \right) ^{
2}}{2({h}^{3}+{h}^{2})}}-\frac{1}{2} \left( c-3 \right)  \left( ch+c-7\,
h-5 \right) +{N_0'} \left( h \right).
$$
Under the hypothesis it now follows that $N_1$ is a polynomial in $h$, so $M_0$ and $M_1$ are polynomials, and therefore $q$ is a polynomial as we wanted.
\end{proof}

\begin{corollary}\label{grau15}
If we take $p=f+h$ and
$$
q=-{\frac { \left( {h
}^{3}+{h}^{2} \right) ^{2}}{{f}^{2}}} -{
\frac {2{h}^{2} \left( {h}^{3}+{h}^{2} \right) }{f}}+4\,{h}^{3}+\frac{3h^2}{2}
 -5hf,
$$
the map $(p,q)$ is a non-injective polynomial local diffeomorphism. The degree of $p$ is $10$ and the degree of $q$ is $15$.
\end{corollary}
\begin{proof}
It suffices to take $c=1$ and $K(h)=3$, hence $N_0(h)=3h^2+2h$, on Theorem \ref{jacobian}. The polynomial $q$ has degree $15$, since the degrees of $h$ and $f$ are $5$ and $10$, respectively.
\end{proof}
\begin{remark}
If we take $(p,q)$ as in Corollary \ref{grau15}, we have
$$
J(p,q)= \left( {\frac {{h}^{3}+{h}^{2}}{f}}+{h}^{2} \right) ^{2}+ \left( {
\frac {{h}^{3}+{h}^{2}}{f}}+3\,{h}^{2}+2\,h-2\,f \right) ^{2}+{f}^{2}.
$$
\end{remark}
\begin{remark}
We point out that by taking $c=0$ and $K(h)=0$ on Theorem \ref{jacobian}, $q$ is the Pinchuk map of degree $25$ that can be found on \cite{vanden,campbell}, so our construction generalizes Pinchuk's construction up to triangular automorphisms (see \cite[Section 2]{campbell}).
\end{remark}

\section{Acknowledgments}
The author thanks Prof. Francisco Braun for introducing the subject and for many suggestions on earlier versions of this paper.
The author was partially financed by the Coordena\c{c}\~ao de Aperfei\c{c}oamento de Pessoal de N\'ivel Superior - Brasil (CAPES) - Finance Code 001.

\end{document}